\documentclass{amsart}

\usepackage{amsmath,amssymb,amsfonts,enumerate,amsthm,graphicx,color,hyperref}

\newcommand{\supp}{\mbox{supp}\,}

\renewcommand{\dim}{\mbox{dim}\,}
\renewcommand{\dim}{\mbox{dim}\,}

\newcommand{\set}{\mbox{set}\,}

\newcommand{\T}{\mathrm}

\newtheorem{thm}{Theorem}[section]
\newtheorem{cor}[thm]{Corollary}
\newtheorem{lem}[thm]{Lemma}

\newtheorem{defn}[thm]{Definition}

\newtheorem{rem}[thm]{Remark}

\def \KK{\mathbb K}
\numberwithin{equation}{section}

\begin{document}
\bibliographystyle{amsplain}

%\date{}
\title[Homological invariants of the Stanley-Reisner ring of ...]{Homological invariants of the Stanley-Reisner ring of a $k$-decomposable simplicial complex}
\author[S. Moradi]{Somayeh Moradi}
%%\thanks{}
%\subjclass{13D02, 13P10, 13D40, 13A02}
\address{Somayeh Moradi, Department of Mathematics,
 Ilam University, P.O.Box 69315-516, Ilam, Iran and School of Mathematics, Institute
 for Research in Fundamental Sciences (IPM), P.O.Box: 19395-5746, Tehran, Iran.} \email{somayeh.moradi1@gmail.com}

\keywords{$k$-decomposable, Alexander dual, regularity, Stanley-Reisner ring.\\
}
\subjclass[2010]{Primary 13D02, 05E45, 05E40; Secondary 16E05}

\begin{abstract}

\noindent We study the regularity and the projective dimension
of the Stanley-Reisner ring of a $k$-decomposable simplicial complex and explain these invariants with a recursive formula.
To this aim, the graded Betti numbers of $k$-decomposable monomial ideals which is the dual concept
for $k$-decomposable simplicial complexes are studied and an inductive formula for the Betti numbers is given.
As a corollary, for a chordal clutter $\mathcal{H}$, an upper bound for $\T{reg}(I(\mathcal{H}))$ is given in terms of the regularities of edge ideals of some chordal clutters which are minors of $\mathcal{H}$.
\end{abstract}

\maketitle
\section*{Introduction}

Squarefree monomial ideals are a class of ideals with strong connections to combinatorics and topology that have been studied extensively by many researchers in the last few years. In this regard the Stanley-Reisner correspondence plays an important role. Finding connections between algebraic invariants of a squarefree monomial ideal and combinatorial and topological invariants of its Stanley-Reisner simplicial complex is of great interest.
$k$-decomposability  is a topological combinatorial concept such as shellability and is related to the
algebraic properties of the Stanley-Reisner ring of a simplicial complex. It was first introduced by Provan and Billera  \cite{Provan+Billera} for pure complexes. An analogous extension was given for non-pure complexes by Woodroofe in \cite{W1}. For $k=0$, this notion is known as vertex decomposable, firstly defined for non-pure complexes in \cite{BW}. Defined in a recursive manner, $k$-decomposable simplicial complexes form a well-behaved class of simplicial complexes.
It is known that a $d$-dimensional simplicial complex is $d$-decomposable if and only if it is shellable (see \cite[Theorem 3.6]{W1}).
In \cite{RY}, the concept of a $k$-decomposable monomial ideal was introduced and it was proved that a simplicial complex $\Delta$ is $k$-decomposable if and only if $I_{\Delta^{\vee}}$ is a $k$-decomposable ideal.

Several recent papers have related some homological invariants of Stanley-Reisner rings such as the Castelnuovo-Mumford regularity and the projective dimension with various invariants of the simplicial complex or the graph associated to the simplicial complex in the case it is a flag complex (see for example \cite{Ha,HTW,KM,kummini,Moh,MVi,VT}).
In this paper, we study the regularity and the projective dimension of a $k$-decomposable monomial ideal and the Stanley-Reisner ring associated to a $k$-decomposable simplicial complex, which extend some existing results about vertex decomposable simplicial complexes.

The paper is structured as follows. In the first section, we present the background
material. In Section 2, first we study the graded Betti numbers of a $k$-decomposable monomial ideal and give an inductive formula for them (see Theorem \ref{ss}). As a corollary, for a $k$-decomposable simplicial complex $\Delta$, the Castelnuovo-Mumford regularity and the projective dimension of the Stanley-Reisner ring $R/I_{\Delta}$ are explained recursively (see Theorem \ref{regp}). This extends the result proved for vertex decomposable simplicial complexes in \cite[Corollary 2.11]{MK} and \cite[Corollary 4.9]{HTW}.  Then for a chordal clutter $\mathcal{H}$ (in the sense of \cite[Definition 4.3]{W1}) with the simplicial vertex $x$ and an edge $e=\{x,x_1,\ldots,x_{d}\}$ containing $x$, we show that
$$\T{reg}(R/I(\mathcal{H}))\leq \max\{\sum_{i=1}^d \T{reg}(R/I(\mathcal{H}\setminus x_i))+(d-1),\T{reg}(R/I(\mathcal{H}/ \{x_1,\ldots,x_d\}))+d\},$$
where  $\mathcal{H}\setminus x_i$'s and $\mathcal{H}/ \{x_1,\ldots,x_d\}$ are chordal minors of $\mathcal{H}$, which are chordal too.

\section{Preliminaries}

In this section, we recall some preliminaries which are needed in the sequel.

\subsection{Simplicial complexes and clutters}

Throughout this paper, we assume that $X=\{x_1, \dots, x_n\}$, $\Delta$ is a simplicial complex on the vertex set $X$, $\KK$ is a field, $R=\KK[X]$ is the ring of polynomials in the variables $x_1, \dots, x_n$ and $I$ is a monomial ideal of $R$. For a monomial ideal $I$,
the unique set of minimal generators of $I$ is denoted by $\mathcal{G}(I)$.

A \textbf{simplicial complex} on a vertex set $X$ is a set $\Delta$  of subsets of $X$ such that

\begin{itemize}
\item for any $x\in X$, $\{x\}\in \Delta$, and
\item  if $F\in \Delta$ and $G\subseteq F$, then $G\in \Delta$.
\end{itemize}

Any element of $\Delta$ is called a \textbf{face} and maximal faces (under inclusion) are called \textbf{facets} of $\Delta$. The set of all facets of a simplicial complex $\Delta$
is denoted by $\mathcal{F}(\Delta)$. A simplicial complex with $\mathcal{F}(\Delta)=\{F_1,\ldots,F_k\}$ is denoted by $\langle F_1,\ldots,F_k\rangle$ and it means that $\Delta$ is generated by the facets $F_1,\ldots,F_k$. The \textbf{dimension} of a face $F\in \Delta$ is defined as $|F|-1$ and is denoted by $\dim(F)$.
Moreover, the dimension of $\Delta$ is defined as $\dim(\Delta)=\max\{\dim(F):\ F\in \Delta\}$.

For a simplicial complex $\Delta$ with the vertex set $X$, the \textbf{Alexander dual simplicial complex} associated to $\Delta$ is defined as
$$\Delta^{\vee}=:\{X\setminus F:\ F\notin \Delta\}.$$

For a subset $W\subseteq X$, let $x^W$ be the monomial $\prod_{x\in W} x$ and let $I=(x^{W_1},\ldots,x^{W_t})$ be a squarefree monomial ideal. The \textbf{Alexander dual ideal} of $I$, denoted by
$I^{\vee}$, is defined as
$$I^{\vee}:=P_{W_1}\cap \cdots \cap P_{W_t},$$
where $P_{W_i}=(x_j:\ x_j\in W_i)$.

One can see that
$$(I_{\Delta})^{\vee}=(x^{F^c} \ : \ F\in \mathcal{F}(\Delta)), $$
where $I_{\Delta}$ is the Stanley-Reisner ideal associated to $\Delta$ and $F^c=X\setminus F$.
Moreover, $(I_{\Delta})^{\vee}=I_{\Delta^{\vee}}$.

%\begin{defn}
%{\rm A simplicial complex $\Delta$ is called \textbf{shellable} if there exists an ordering $F_1<\cdots<F_m$ on the
%facets of $\Delta$
%such that for any $i<j$, there exists a vertex
%$v\in F_j\setminus F_i$ and  $\ell<j$ with
%$F_j\setminus F_\ell=\{v\}$. We call $F_1,\ldots,F_m$ a \textbf{shelling order} for
%$\Delta$.}
%\end{defn}
%The above definition is referred to as  non-pure shellable and
%is due to Bj\"{o}rner and Wachs \cite{BW}. In this paper we will
%drop the adjective ``non-pure".

A \textbf{clutter} $\mathcal{H}$ on a vertex set $V$ is a set $E$ of subsets of $V$ (called edges) such that
\begin{itemize}
  \item  for any distinct elements $e,e'\in E$,
one has $e\nsubseteq e'$, and
  \item $|e|\geq 2$ for any $e\in E$.
\end{itemize}
The vertex set and the edge set of the clutter $\mathcal{H}$ are denoted by $V(\mathcal{H})$ and $E(\mathcal{H})$, respectively.
There is a correspondence between clutters and the set of minimal
non-faces of simplicial complexes in the following way.

Let $\mathcal{H}$ be a clutter. The \textbf{edge ideal} of $\mathcal{H}$ is an ideal of the ring $\KK[V(\mathcal{H})]$ defined as  $$I(\mathcal{H})=(x^e:\ e\in E(\mathcal{H})).$$ It is easy to see that $I(\mathcal{H})$ can be viewed as the Stanley-Reisner ideal of the simplicial complex
$$\Delta_{\mathcal{H}}=\{F\subseteq V(\mathcal{H}):\ e\nsubseteq F, \text{ for each } e\in E(\mathcal{H})\},$$ i.e., $I(\mathcal{H})=I_{\Delta_{\mathcal{H}}}$. The simplicial complex $\Delta_{\mathcal{H}}$ is called the \textbf{independence complex} of $\mathcal{H}$.

For a clutter $\mathcal{H}$, and a vertex $v\in V(\mathcal{H})$, the \textbf{deletion} $\mathcal{H}\setminus v$ is the clutter on the vertex set $V (\mathcal{H}) \setminus \{v\}$ with the edge set
$\{e :\ e \in E(\mathcal{H}),\ v\notin  e\}$. The \textbf{contraction} $\mathcal{H}/v$ is the clutter on the vertex set
$V (\mathcal{H})\setminus\{v\}$ with edges the minimal sets of $\{e\setminus\{v\} :\ e \in E(\mathcal{H})\}$.
A clutter $\mathcal{K}$ obtained from $\mathcal{H}$ by repeated deletion and/or contraction is called a \textbf{minor} of $\mathcal{H}$.

The concept of a simplicial vertex and a chordal clutter was introduced in \cite[Definitions 4.2 and 4.3]{W1} as follows.

Let $\mathcal{H}$ be a clutter. A vertex $v$ of $\mathcal{H}$ is called a \textbf{simplicial vertex} if for every two edges
$e_1$ and $e_2$ of $\mathcal{H}$ that contain $v$, there is a third edge $e_3$ such that $e_3 \subseteq (e_1 \cup e_2)\setminus \{v\}$.

A clutter $\mathcal{H}$ is called \textbf{chordal}, if any minor of $\mathcal{H}$ has a simplicial vertex.

Woodroofe in \cite{W1} also defined the concept of a
\textbf{containment pair} of a clutter $\mathcal{H}$  as a vertex $v$ and an edge $e$ with $v\in e$ such that
for any edge $e_2\neq e$ with  $v\in e_2$, there exists an edge $e_3\subseteq (e\cup e_2) \setminus \{v\}$.

\subsection{$k$-decomposable simplicial complexes and $k$-decomposable ideals}

\vskip 3mm

For a simplicial complex $\Delta$ and $F\in \Delta$, the link of $F$ in
$\Delta$ is defined as $$\T{lk}(F)=\{G\in \Delta: G\cap
F=\emptyset, G\cup F\in \Delta\},$$ and the deletion of $F$ is the
simplicial complex $$\Delta \setminus F=\{G\in \Delta: F \nsubseteq G\}.$$

Woodroofe in \cite{W1} extended the definition of $k$-decomposability
to non-pure complexes as follows.

\begin{defn}\cite[Definition 3.1]{W1}
{\rm Let $\Delta$ be a simplicial complex on vertex set $V$. Then a face $\sigma$ is called a
\textbf{shedding face} if every face $\tau$  containing $\sigma$ satisfies the following exchange property: for every
$v \in \sigma$ there is $w\in V \setminus \tau$ such that $(\tau \cup \{w\})\setminus \{v\}$ is a face of $\Delta$.}
\end{defn}

\begin{defn}\cite[Definition 3.5]{W1}
{\rm A simplicial complex $\Delta$ is recursively defined to be \textbf{$k$-decomposable} if
either $\Delta$ is a simplex or else has a shedding face $\sigma$ with $\dim(\sigma)\leq k$ such that both $\Delta \setminus \sigma$
and $\T{lk}(\sigma)$ are $k$-decomposable. The complexes $\{\}$ and $\{\emptyset\}$  are considered to be
$k$-decomposable for all $k \geq -1$.}
\end{defn}

Note that $0$-decomposable simplicial complexes are precisely vertex decomposable simplicial complexes.

The notion of decomposable monomial ideals was introduced in \cite{RY} as follows.

For the monomial $u=x_1^{a_1}\cdots x_n^{a_n}$ in $R$, the support of $u$  denoted by $\supp(u)$ is the set $\{x_i:\ a_i\neq 0\}$. For a monomial $M$ in $R$,
set $[u,M] = 1$ if for all $x_i\in\supp(u)$, $x_i^{a_i}\nmid M$. Otherwise set $[u,M]\neq 1$.

For the monomial $u$ and the monomial ideal $I$, set
$$I^u = (M\in \mathcal{G}(I) :\ [u,M]\neq 1)$$ and $$I_u = (M\in \mathcal{G}(I) :\ [u,M]=1).$$

\begin{defn}
{\rm  For a monomial ideal $I$  with $\mathcal{G}(I)=\{M_1,\ldots,M_r\}$, the monomial $u=x_1^{a_1}
\cdots x_n^{a_n}$ is called  a \textbf{shedding monomial} for $I$ if $I_u\neq 0$ and for each $M_i\in \mathcal{G}(I_u)$ and each
$x_{\ell}\in \supp(u)$ there exists $M_j\in \mathcal{G}(I^u)$ such that $M_j:M_i=x_{\ell}$.}
\end{defn}

\begin{defn}
{\rm
A monomial ideal $I$ with $\mathcal{G}(I)=\{M_1,\ldots,M_r\}$ is called  \textbf{$k$-decomposable} if $r=1$ or else has a shedding monomial $u$ with
$|\supp(u)|\leq k+1$ such that the ideals $I_u$ and $I^u$ are $k$-decomposable.}
\end{defn}

\begin{defn}\label{1.2}
{\rm
A monomial ideal $I$ in the ring $R$ has \textbf{linear quotients} if there exists an ordering $f_1, \dots, f_m$ on the minimal generators of $I$ such that the colon ideal $(f_1,\ldots,f_{i-1}):(f_i)$ is generated by a subset of $\{x_1,\ldots,x_n\}$ for all $2\leq i\leq m$. We show this ordering by $f_1<\dots <f_m$ and we call it an order of linear quotients on $\mathcal{G}(I)$.

Let $I$ be a monomial ideal with linear quotients and $f_1<\dots <f_m$ be an order of linear quotients on the minimal generators of $I$. For any $1\leq i\leq m$, $\set_I(f_i)$ is defined as
$$\set_I(f_i)=\{x_k:\ x_k\in (f_1,\ldots, f_{i-1}) : (f_i)\}.$$
}
\end{defn}

%For a $\mathbb{Z}$-graded $R$-module $M$, the \textbf{Castelnuovo-Mumford regularity} (or simply regularity)
%of $M$ is defined as
%$$\T{reg}(M) := \max\{j-i: \ \beta_{i,j}(M)\neq 0\},$$
%and the \textbf{projective dimension} of $M$ is defined as
%$$\T{pd}(M) := \max\{i:\ \beta_{i,j}(M)\neq 0 \ \text{for some}\ j\},$$
%where $\beta_{i,j}(M)$ is the $(i,j)$-th graded Betti number of $M$.

The following theorem, which was proved in \cite{T}, is one of our
main tools in the study of regularity of the ring $R/I_{\Delta}$.

\begin{thm}\cite[Theorem 2.1]{T} \label{1.3}
Let $I$ be a squarefree monomial ideal. Then
$\T{pd}(I^{\vee})=\T{reg}(R/I)$.
\end{thm}

\section{Regularity and projective dimension of the Stanley-Reisner ring of a $k$-decomposable simplicial complex}

In this section, first we explain the graded Betti numbers of a $k$-decomposable ideal inductively. Then using this formula, for a $k$-decomposable simplicial compelx $\Delta$, recursive formulas for the regularity and the projective dimension of the Stanley-Reisner ring $R/I_{\Delta}$ are given.
Finally, we give an applicaion on chordal clutters.

The following theorem was proved in \cite{RY}. In order to prove Theorem \ref{ss}, one needs to know the construction of the order of linear quotients for a decomposable ideal. So for the convenience of the reader, the proof is stated.
\begin{thm}\cite[Theorem 2.13]{RY}\label{q}
Any $k$-decomposable ideal has linear quotients.
\end{thm}

\begin{proof}
Let $I$ be a $k$-decomposable ideal, $u$ be a shedding monomial for $I$ and inductively assume that $I_u$ and $I^u$ have linear quotients.
Let $f_1<\cdots<f_t$ and $g_{t+1}<\cdots<g_r$ be the order of linear quotients on the minimal generators of $I^u$ and $I_u$, respectively. Then
it is easy to see that
$f_1<\cdots<f_t<g_{t+1}<\cdots<g_r$ is an order of linear quotients on the minimal generators of $I$, since for any $t+1\leq i\leq r$, $(f_1,\ldots,f_t,g_{t+1},\ldots,g_{i-1}):(g_i)=(x_i:\ x_i\in\supp(u))+(g_{t+1},\ldots,g_{i-1}):(g_i)$, which is again generated by some variables.
Indeed, by the definitions of $I_u$ and $I^u$, for any $x\in \supp(u)$ and any $t+1\leq j\leq r$, there exists $1\leq i\leq t$ such that $(f_i):(g_j)=(x)$.
\end{proof}

The next theorem, which is a special case of \cite[Corollary 2.7]{Leila}, is our main tool to prove Theorem \ref{ss}.

\begin{thm}\cite[Corollary 2.7]{Leila}\label{Leila}
Let $I$ be a monomial ideal with linear quotients with the ordering $f_1<\cdots<f_m$ on the minimal generators of $I$.
Then $$\beta_{i,j}(I)=\sum_{\deg(f_t)=j-i} {|\set_I(f_t)|\choose i}.$$
\end{thm}

We also use the following lemma, which can be easily proved by induction on $m$.
\begin{lem}\label{lem1}
Let $i,k$ and $m$ be non-negative integers. Then $${k+m\choose i}=\sum_{\ell=0}^m {m\choose \ell}{k\choose i-\ell}.$$
\end{lem}
%\begin{proof}
%We use induction on $m$. If $m=1$, then clearly $${k+1\choose i}={k\choose i}+{k\choose i-1}=\sum_{\ell=0}^1 {1\choose \ell}{k\choose i-\ell}.$$
%Assume inductively that $${k+m-1\choose i}=\sum_{\ell=0}^{m-1} {m-1\choose \ell}{k\choose i-\ell}.$$
%Then $$\sum_{\ell=0}^m {m\choose \ell}{k\choose i-\ell}=\sum_{\ell=0}^{m} [{m-1\choose \ell}+{m-1\choose \ell-1}]{k\choose %i-\ell}=\sum_{\ell=0}^{m-1} {m-1\choose \ell}{k\choose i-\ell}+$$ $$\sum_{\ell=1}^m {m-1\choose \ell-1}{k\choose i-\ell}={k+m-1\choose %i}+{k+m-1\choose i-1}={k+m\choose i}.$$
%\end{proof}

The following theorem generalizes \cite[Theorem 2.8]{MK}.
\begin{thm}\label{ss}
Let $I$ be a $k$-decomposable ideal with the shedding monomial $u$. Then $$\beta_{i,j}(I)=\beta_{i,j}(I^u)+\sum_{l=0}^m {m\choose l} \beta_{i-l,j-l}(I_u),$$ where $m=|\supp(u)|$.
\end{thm}

\begin{proof}
For a monomial $f= x_1^{a_1}\cdots x_n^{a_n}$ and $1\leq i\leq n$,  we set $v_i(f)=a_i$. By Theorem \ref{q}, if
$f_1<\cdots<f_t$ is an order of linear quotients on the minimal generators of $I^u$   and $g_{t+1}<\cdots<g_r$ is an order of linear quotients on the minimal generators of $I_u$, then $f_1<\cdots<f_t<g_{t+1}<\cdots<g_r$ is an order of linear quotients on the minimal generators of $I$,
$\set_I(f_i)=\set_{I^u}(f_i)$ for any $1\leq i\leq t$ and $\set_I(g_i)=\supp(u)\bigcup \set_{I_u}(g_i)$ for any $t+1\leq i\leq r$. Also for any $t+1\leq i \leq r$, $\supp(u)\bigcap \set_{I_u}(g_i)=\emptyset$, since for any $x_{\ell}\in \supp(u)$ and any distinct integers $i$ and $j$,  $v_{\ell}(g_i)=v_{\ell}(g_j)=v_{\ell}(u)-1$ and then $x_{\ell}\nmid (g_i):(g_j)$.  Thus $|\set_I(g_i)|=|\set_{I_u}(g_i)|+m$. Now by Theorem \ref{Leila},
$$\beta_{i,j}(I)=\sum_{\deg(f_s)=j-i}{|\set_{I}(f_s)|\choose i}+
\sum_{\deg(g_s)=j-i}{|\set_{I}(g_s)|\choose i}.$$
Thus
$$\beta_{i,j}(I)=\sum_{\deg(f_s)=j-i}{|\set_{I^{u}}(f_s)|\choose i}+
\sum_{\deg(g_s)=j-i}{|\set_{I_{u}}(g_s)|+m\choose i}.$$
Applying Lemma \ref{lem1}, we have
$${|\set_{I_{u}}(g_s)|+m\choose i}=\sum_{\ell=0}^m{m \choose \ell}{|\set_{I_{u}}(g_s)|\choose i-\ell}.$$
Thus
$$\beta_{i,j}(I)=\beta_{i,j}(I^u)+\sum_{\ell=0}^m{m \choose \ell}\beta_{i-\ell,j-\ell}(I_u).$$
\end{proof}

In the following corollary, recursive formulas for the regularity and the projective dimension of a $k$-decomposable ideal are presented.

\begin{cor}\label{corvd}
Let $I$ be a $k$-decomposable ideal with the shedding monomial $u$ and $m=|\supp(u)|$. Then
\begin{itemize}
 \item [(i)]  $\T{pd}(I)=\max\{\T{pd}(I^u),\T{pd}(I_u)+m\}$, and
  \item [(ii)] $\T{reg}(I)=\max\{\T{reg}(I^u),\T{reg}(I_u)\}$.
\end{itemize}
\end{cor}

H\`{a} in \cite{Ha} gave an upper bound for the regularity of $R/I_{\Delta}$ for an arbitrary simplicial complex $\Delta$ in terms of the deletion and the link of a face of $\Delta$ as follows.

\begin{thm}\label{ha}\cite[Theorem 3.4]{Ha}
Let $\Delta$ be a simplicial complex and let $\sigma$ be a face of $\Delta$. Then
\begin{equation}
\T{reg}(R/I_{\Delta})\leq\max \{\T{reg}(R/I_{\Delta\setminus\sigma}),\T{reg}(R/I_{\T{lk}(\sigma)})+|\sigma|\}.
\end{equation}
\end{thm}

As one of the main results of this paper, in Theorem \ref{regp}, it is shown that the inequality in Theorem \ref{ha} becomes an equality for a $k$-decomposable simplicial complex and a shedding face $\sigma$. It also generalizes \cite[Corollary 2.11]{MK}. To this aim we use the following theorem.

\begin{thm}\cite[Theorem 2.10]{RY}
A $d$-dimensional simplicial complex $\Delta$ is $k$-decomposable if and only if $I_{\Delta^{\vee}}$ is a squarefree $k$-decomposable ideal for any $k\leq d$.
\end{thm}

\begin{thm}\label{regp}
Let $\Delta$ be a $k$-decomposable simplicial complex on the vertex set $X$ with the shedding face $\sigma$. Then
\begin{itemize}
  \item [(i)] $\T{reg}(R/I_{\Delta})=\max \{\T{reg}(R/I_{\Delta\setminus\sigma}),\T{reg}(R/I_{\T{lk}(\sigma)})+|\sigma|\}$,
  \item [(ii)] $\T{pd}(R/I_{\Delta})=\max\{\T{pd}(R/I_{\Delta\setminus\sigma}),\T{pd}(R/I_{\T{lk}(\sigma)})\},$
\end{itemize}
where $I_{\Delta\setminus\sigma}$ and $I_{\T{lk}(\sigma)}$ are Stanley-Reisner ideals of $\Delta\setminus\sigma$ and $\T{lk}(\sigma)$
on the vertex sets $X$ and  $X\setminus \sigma$, respectively.
\end{thm}

\begin{proof}
$(i)$ By \cite[Theorem 2.10]{RY}, $I_{\Delta^{\vee}}$ is a decomposable ideal and $x^{\sigma}$ is a shedding monomial for $I_{\Delta^{\vee}}$.  Also by Theorem \ref{1.3}, $\T{reg}(R/I_{\Delta})=\T{pd}(I_{\Delta^{\vee}})$. Now, using Corollary \ref{corvd},
$$\T{pd}(I_{\Delta^{\vee}})=\max\{\T{pd}((I_{\Delta^{\vee}})^u),\T{pd}((I_{\Delta^{\vee}})_u)+|\sigma|\}.$$

Let $\Delta=\langle F_1,\ldots,F_m\rangle$ and $\sigma=\{x_1,\ldots,x_t\}$. Then $I_{\Delta^{\vee}}=(x^{F_1^c},\ldots,x^{F_m^c})$, where $F_i^c=X\setminus F_i$.
Thus $(I_{\Delta^{\vee}})^u=(x^{F_i^c}:\ x_j| x^{F_i^c}  \ \textrm{for some}\ 1\leq j\leq t)=(x^{F_i^c}:\  \sigma\nsubseteq F_i)$. Since for any $F_i$ with $\sigma\subseteq F_i$ and any $1\leq l\leq t$, there exists $F_j$ such that $\sigma\nsubseteq F_j$ and $x^{F_j^c}:x^{F_i^c}=x_l$ or equivalently $F_i\setminus \{x_l\}\subseteq  F_j$, so for any face $F\in \Delta$ with $\sigma\nsubseteq F$, one has $F\subseteq F_i\setminus \{x_l\}$ for some $1\leq l\leq t$ and some $1\leq i\leq m$ then $F\subseteq F_j$ for some $F_j$ not containing $\sigma$. Thus $\Delta\setminus \sigma=\langle F_j:\ \sigma\nsubseteq F_j\rangle$ and then $(I_{\Delta^{\vee}})^u=I_{(\Delta\setminus \sigma)^{\vee}}$. Also $(I_{\Delta^{\vee}})_u=(x^{F_i^c}:\ x_j\nmid x^{F_i^c}  \ \textrm{for any}\ 1\leq j\leq t)=(x^{F_i^c}:\  \sigma\subseteq F_i)$. Since for any $F_i$ with $\sigma\subseteq F_i$, one has $X\setminus F_i=(X\setminus \sigma)\setminus (F_i\setminus \sigma)$, so $(I_{\Delta^{\vee}})_u=I_{\T{lk}(\sigma)^{\vee}}$, where $I_{\T{lk}(\sigma)}$ is the Stanley-Reisner ideal of $\T{lk}(\sigma)$ on the vertex set $X\setminus \sigma$.
Thus $$\T{pd}(I_{\Delta^{\vee}})=\max\{\T{pd}(I_{(\Delta\setminus \sigma)^{\vee}}),\T{pd}(I_{\T{lk}(\sigma)^{\vee}})+|\sigma|\}.$$
Again using Theorem \ref{1.3}, the result follows.

$(ii)$ By Corollary \ref{corvd},
$$\T{reg}(I_{\Delta^{\vee}})=\max\{\T{reg}((I_{\Delta^{\vee}})^u),\T{reg}((I_{\Delta^{\vee}})_u)\}.$$ As was discussed above,
$(I_{\Delta^{\vee}})^u=I_{(\Delta\setminus \sigma)^{\vee}}$ and $(I_{\Delta^{\vee}})_u=I_{\T{lk}(\sigma)^{\vee}}$. Thus
$$\T{reg}(I_{\Delta^{\vee}})=\max\{\T{reg}(I_{(\Delta\setminus \sigma)^{\vee}}),\T{reg}(I_{\T{lk}(\sigma)^{\vee}})\}.$$
From the equalities $\T{pd}(R/I_{\Delta})=\T{pd}(I_{\Delta})+1$, $(I_{\Delta^{\vee}})^{\vee}=I_{\Delta}$, $\T{reg}(R/I_{\Delta^{\vee}})=\T{reg}(I_{\Delta^{\vee}})-1$ and  Theorem \ref{1.3}, one has  $\T{pd}(R/I_{\Delta})=\T{reg}(I_{\Delta^{\vee}})$, which completes the proof.
\end{proof}

\begin{rem}
Any $k$-decomposable simplicial complex is shellable and hence sequentially Cohen-Macaulay. So by \cite[Corollary 3.33]{MVi}, for a  $k$-decomposable simplicial complex $\Delta$, $\T{pd}(R/I_{\Delta})$ is equal to the big height of $I_{\Delta}$.
\end{rem}

\begin{rem} (Compare \cite[Corollary 2.11(ii)]{MK})
{\rm If $\Delta$ is a vertex decomposable simplicial complex on the set $X$ with the shedding vertex $x$ , then $I_{(\Delta\setminus \{x\})^{\vee}}=xJ_{(\Delta\setminus \{x\})^{\vee}}$, where $I_{\Delta\setminus\{x\}}$ and $J_{\Delta\setminus\{x\}}$ are Stanley-Reisner ideals of $\Delta\setminus\{x\}$
on the vertex sets $X$ and  $X\setminus \{x\}$, respectively. Thus $\T{pd}(R/I_{\Delta\setminus\{x\}})=\T{reg}(I_{(\Delta\setminus \{x\})^{\vee}})=\T{reg}(J_{(\Delta\setminus \{x\})^{\vee}})+1=\T{pd}(R/J_{\Delta\setminus\{x\}})+1$, where $R=\KK[X]$. Thus $$\T{pd}(R/I_{\Delta})=\max\{\T{pd}(R/J_{\Delta\setminus \{x\}})+1,\T{pd}(R/I_{\T{lk}(x)})\}.$$ Note that in \cite[Corollary 2.11]{MK},
$\Delta\setminus\{x\}$ and $\T{lk}(x)$ are considered on the vertex set $X\setminus \{x\}$.
}
\end{rem}

%The next result gives a formula for the regularity of the Stanley-Reisner ring of a shellable simplicial complex $\Delta$ in terms of a shelling %order of $\Delta$.

%\begin{thm}\label{she}
%Let $\Delta$ be a shellable simplicial complex with the shelling order $F_1<\cdots<F_k$ and $\dim(\Delta)=d$. For any $1\leq i\leq k$, let %$\Delta_i=\langle F_1,\ldots,F_i\rangle$ and $\mathcal{R}(F_i)=\{x\in F_i:\ F_i\setminus \{x\}\in \Delta_{i-1}\}$. Then
%$$\T{reg}(R/I_{\Delta})=\max\{|\mathcal{R}(F_1)|,\ldots,|\mathcal{R}(F_k)|\}.$$

%\end{thm}
%\begin{proof}
%By \cite[Theorem 3.6]{W1}, $\Delta$ is $d$-decomposable and $\mathcal{R}(F_k)$ is a shedding face for $\Delta$. Thus by \cite[Theorem 2.10]{RY} %$u=x^{\mathcal{R}(F_k)}$ is a shedding monomial for $I_{\Delta^{\vee}}$. Thus by Theorem \ref{regp},
%\begin{equation}\label{e1}
%\T{reg}(R/I_{\Delta})=\max\{\T{reg}(R/I_{\Delta\setminus \mathcal{R}(F_k)}),\T{reg}(R/I_{\T{lk}(\mathcal{R}(F_k))})+|\mathcal{R}(F_k)|\}.
%\end{equation}
%Note that $$\Delta\setminus \mathcal{R}(F_k)=\Delta_{k-1},$$ which is shellable with the shelling order $F_1<\cdots<F_{k-1}$
%and $\T{lk}(\mathcal{R}(F_k))=\langle F_k\setminus \mathcal{R}(F_k)\rangle$, which is shellable too. By induction assume that
%\begin{equation}\label{e2}
%\T{reg}(R/I_{\Delta_{k-1}})=\max\{|\mathcal{R}(F_1)|,\ldots,|\mathcal{R}(F_{k-1})|\}.
%\end{equation}
%One has $I_{\T{lk}(\mathcal{R}(F_k))}=0$, since $\T{lk}(\mathcal{R}(F_k))$ is a simplex. Thus $\T{reg}(R/I_{\T{lk}(\mathcal{R}(F_k))})=0$. Now %by \ref{e1} and \ref{e2}, the result holds.
%\end{proof}

As an application, we give an upper bound for the regularity of $R/I(\mathcal{H})$ for a chordal clutter $\mathcal{H}$ in terms of the regularities of edge ideals of some minors of $\mathcal{H}$ which are chordal too. First we state the following easy lemma.

\begin{lem}\label{h}
Let $\mathcal{H}$ be a clutter, $\Delta=\Delta_{\mathcal{H}}$, $e=\{x,x_1,\ldots,x_{d}\}\in E(\mathcal{H})$ and $\sigma=e\setminus \{x\}$. Then $I_{\Delta\setminus \sigma}=(x_1\cdots x_d)+I(\mathcal{H}\setminus x_1)+\cdots+I(\mathcal{H}\setminus x_d)$ and $I_{\T{lk}(\sigma)}=I(\mathcal{H}/ \{x_1,\ldots,x_{d}\})$.
\end{lem}

\begin{proof}
Note that $I_{\Delta\setminus \sigma}=(x^F:\ F\in \mathcal{N}(\Delta\setminus \sigma))$, where $\mathcal{N}(\Delta\setminus \sigma)$ denotes the set of minimal non-faces of $\Delta\setminus \sigma$. One has $F\in \mathcal{N}(\Delta\setminus \sigma)$ if and only if either $F\in \mathcal{N}(\Delta)$ or $F=\sigma$. Thus $I_{\Delta\setminus \sigma}=(x_1\cdots x_{d})+I_{\Delta}=(x_1\cdots x_{d})+I(\mathcal{H})$. For any edge $e'\in E(\mathcal{H})$, if $\{x_1,\ldots,x_{d}\}\subseteq e'$, then $x^{e'}\in (x_1\cdots x_{d})$. Otherwise $x^{e'}\in I(\mathcal{H}\setminus x_i)$ for some $1\leq i\leq d$. Thus
$$(x_1\cdots x_{d})+I(\mathcal{H})=(x_1\cdots x_d)+I(\mathcal{H}\setminus x_1)+\cdots+I(\mathcal{H}\setminus x_d).$$
Also $I_{\T{lk}(\sigma)}=(x^F:\ F\subseteq V(\mathcal{H})\setminus \sigma,\ F\notin \T{lk}(\sigma))$. For any $F\subseteq V(\mathcal{H})\setminus \sigma$, one has $$F\notin \T{lk}(\sigma) \Leftrightarrow F\cup \sigma\notin \Delta \Leftrightarrow \exists\ e\in E(\mathcal{H})\ \ e\setminus \sigma\subseteq F.$$
Thus $I_{\T{lk}(\sigma)}=I(\mathcal{H}/ \{x_1,\ldots,x_{d}\})$.
\end{proof}

\begin{thm}
Let $\mathcal{H}$ be a chordal clutter, $x\in V(\mathcal{H})$ be a simplicial vertex for $\mathcal{H}$ and $e=\{x,x_1,\ldots,x_{d}\}$ be an edge of $\mathcal{H}$ containing $x$ and $\sigma=e\setminus \{x\}$. Then
\begin{itemize}
  \item [(i)] $\T{reg}(R/I(\mathcal{H}))=\max \{\T{reg}(R/((x_1\cdots x_d)+I(\mathcal{H})),\T{reg}(R/I(\mathcal{H}/ \{x_1,\ldots,x_{d}\}))+d\}$, and
  \item [(ii)] $\T{reg}(R/I(\mathcal{H}))\leq \max\{\sum_{i=1}^d \T{reg}(R/I(\mathcal{H}\setminus x_i))+(d-1),\T{reg}(R/I(\mathcal{H}/ \{x_1,\ldots,x_d\}))+d\}.$
\end{itemize}
\end{thm}
\begin{proof}
$(i)$ Since $x$ is a simplicial vertex, one may observe that $(x,e)$ is a neighborhood containment pair of $\mathcal{H}$.  Thus by \cite[Lemma 5.1]{W1}, $\sigma=e\setminus \{x\}$ is a shedding face of $\Delta=\Delta_{\mathcal{H}}$ and by \cite[Corollary 5.4]{W1}, $\Delta$ is $(k-2)$-decomposable, where $k$ is the maximum cardinality of edges in $\mathcal{H}$. Thus by Theorem \ref{regp}, $\T{reg}(R/I_{\Delta})=\max \{\T{reg}(R/I_{\Delta\setminus\sigma}),\T{reg}(R/I_{\T{lk}(\sigma)})+|\sigma|\}$.
Now, by Lemma \ref{h}, the result holds.

$(ii)$ Using the equality $(x_1\cdots x_{d})+I(\mathcal{H})=(x_1\cdots x_d)+I(\mathcal{H}\setminus x_1)+\cdots+I(\mathcal{H}\setminus x_d)$,  part $(i)$, \cite[Theorem 1.4]{kmu} and the fact that $\T{reg}(R/(x_1\cdots x_d))=d-1$, the result follows.
\end{proof}

%\textbf{Acknowledgments:}
%The authors would like to thank the referee for his/her valuable comments which substantially improved the quality of the paper.
%The research of the first author was in part supported by a grant from IPM (No. 91130021).

\providecommand{\bysame}{\leavevmode\hbox
to3em{\hrulefill}\thinspace}

\end{document}